\documentclass[a4paper]{article}

\usepackage{graphicx,amsmath,amssymb,enumerate,array}
\usepackage[margin=2.5cm,nohead]{geometry}
\usepackage[latin2]{inputenc}
\usepackage[T1]{fontenc}
\usepackage{latexsym}
\usepackage{amsfonts}
\usepackage{amssymb}
\usepackage{hyperref}
\usepackage{hyperref}
\usepackage{amsthm}
\usepackage{lmodern}
\usepackage{indentfirst}
\usepackage{tabularx}

\newcommand{\vecc}[1]{\mathbf{#1}}
\newcommand{\problem}[1]{
\noindent
\textsc{#1}
}
\newcommand{\modu}[1]{(\textnormal{mod }#1)}

\newtheorem{masttheo}{MastTheorem}

\newtheorem{theo}[masttheo]{Theorem}

\newtheorem{de}[masttheo]{Definition}
\newtheorem{lemma}[masttheo]{Lemma}
\newtheorem{cor}[masttheo]{Corollary}

\newtheorem{remark}[masttheo]{Remark}
\newtheorem{example}[masttheo]{Example}

\begin{document}

\title{Combinatorial Nullstellensatz modulo prime powers \\ and the Parity Argument}
\author{L\'aszl\'o Varga \\  {\small Institute of Mathematics, E\"otv\"os Lor\'and University, Budapest, \href{mailto:LVarga@cs.elte.hu}{\url{LVarga@cs.elte.hu}}}}

\maketitle

\begin{abstract}
We present new generalizations of Olson's theorem and of a consequence of Alon's Combinatorial Nullstellensatz. These enable us to extend some of their combinatorial applications with conditions modulo primes to conditions modulo prime powers.
We analyze computational search problems corresponding to these kinds of combinatorial questions and we prove that the problem of finding degree-constrained subgraphs modulo $2^d$ such as $2^d$-divisible subgraphs and the search problem corresponding to the Combinatorial Nullstellensatz over $\mathbb{F}_2$ belong to the complexity class Polynomial Parity Argument (PPA). 
\end{abstract}

\section{Introduction}

In this paper, we are interested in combinatorial and computational problems in connection with Alon's Combinatorial Nullstellensatz (\cite{alon}) which is a landmark theorem in algebraic combinatorics.

\begin{theo}[Combinatorial Nullstellensatz, Alon, \cite{alon}]
Let $\mathbb{F}$ be an arbitrary field, and let $f\in \mathbb{F}[x_1,\dots x_m]$ be an $m$-variable polynomial. Suppose that the degree of $f$ is $\sum_{j=1}^n t_j$, where each $t_j$ is a nonnegative integer, and that the coefficient of $\prod_{j=1}^m x_j^{t_j}$ is nonzero. Then, if $S_1,S_2,\dots,S_m$ are subsets of $\mathbb{F}$ with $|S_j| > t_j$ for all $j=1,\dots,m$, then there exists an  $(s_1,s_2,\dots, s_m) \in S_1 \times S_2 \times \dots \times S_m$ such that $f(s_1,s_2,\dots, s_m)\neq 0$.
\end{theo}

The following corollary is often used implicitly in applications, see \cite{alon}.

%In these applications, the Combinatorial Nullstellensatz guarantees the existence of a nontrivial object with some prescribed property.

\begin{cor} \label{combnullprimes}
Let $p$ be an arbitrary prime. Let us be given some $m$-variable polynomials $f_1,f_2,\ldots,f_n$ over $\mathbb{F}_p$ with no constant terms and $Q_1,Q_2,\ldots,Q_n\subseteq \mathbb{F}_p$ such that $0 \in Q_i$ for all $i$. If
\[
m>\sum_{i=1}^n \deg(f_i)\cdot |\mathbb{F}_{p} \backslash Q_i|,
\]
then there exists a vector $\vecc{0}\neq \vecc{x}\in \{0,1\}^m$ such that $f_i(\vecc{x})\in Q_i$ for all $i$. 
\end{cor}

%Via Combinatorial Nullstellensatz it is easy to give an estimation for the value $m_0$ depending on degree of polynomials $f_i$ and size of sets $Q_i$. Precisely, $m_0 \geq \sum_{i=1}^n \deg(f_i)\cdot |\mathbb{F}_{p} \backslash Q_i|$ follows from the Nullstellensatz.
\begin{proof}
Let $f(\vecc{x})=\prod_{i=1}^n \prod_{q \not \in Q_i} (q-f_i(\vecc{x})) - c \cdot \prod_{j=1}^m (1-x_j)$ over $\mathbb{F}_p$, where $c=\prod_{i=1}^n \prod_{q \not \in Q_i} q$. It is easy to check that $\deg(f)=m>\sum_{i=1}^n \deg(f_i)\cdot |\mathbb{F}_{p} \backslash Q_i|$ and for a vector $\vecc{x}\in \{0,1\}^m$, $f(\vecc{x})\neq 0$ if and only if $\vecc{0}\neq \vecc{x}$ and  $f_i(\vecc{x})\in Q_i$ for all $i$. Then, with setting $S_i=\{0,1\}$ for all $i$, the Combinatorial Nullstellensatz implies the statement.
\end{proof}

\bigskip 

The goal of this paper is to give similar theorems for problems modulo arbitrary prime powers: we prove that if the number $m$ of variables is sufficiently large, the corollary also holds modulo arbitrary prime powers. We develop a general method for the Combinatorial Nullstellensatz-type proofs, where the polynomials are modulo prime powers instead of primes. As an application, we extend the following theorem of Olson (\cite{ols}) and its generalization by Alon, Friedland and Kalai \cite{reg}.

Let us be given a prime $p$, nonnegative integers $d_1\geq d_2 \geq \dots \geq d_n$ and sets $Q_1,Q_2,\dots,Q_n$  such that each of them contains zero and $Q_i \subseteq\mathbb{Z}_{p^{d_i}}$ for every $i=1,\dots ,n$. Let us denote $(d_1,d_2,\dots, d_n)$ by $\mathbf{d}$ and  $(Q_1,Q_2,\dots, Q_n)$ by $\mathbf{Q}$.

Alon et al. \cite{reg} ask to determine the minimum value $F(\mathbf{d},\mathbf{Q})$ such that for every $m>F(\mathbf{d},\mathbf{Q})$ and for arbitrary integers $a_{ij}$ $\left(i=1,\dots ,n,j=1,\dots ,m\right)$ there exists a nonempty subset $J \subseteq \{1,2,\dots,m \}$ that fulfills the following condition:
\[
\sum_{j\in J} a_{ij} \equiv q_i \qquad \modu{p^{d_i}} \quad \textnormal{ for some } q_i \in Q_i \textnormal{ for every }i=1,\dots ,n. \label{eq:cong}
\tag{$\clubsuit$}
\]

Using this terminology, we can easily formulate Olson's theorem and its extension by Alon et al. as follows:

\begin{theo}[Olson, \cite{ols}] \label{olson}
$F(\mathbf{d},\mathbf{Q}) = \sum_{i=1}^n \left(p^{d_i} - 1\right)$, if $\{0\}=Q_i$ for all $i$.
\end{theo}

\begin{theo}[Alon, Friedland, Kalai, \cite{reg}] \label{alonolson}
$F(\mathbf{d},\mathbf{Q}) \leq \sum_{i=1}^n \left(p^{d_i} - card_p(Q_i)\right)$ where $card_p(Q)$ denotes the number of distinct elements in $Q$ modulo $p$.
\end{theo}

Whereas Theorem \ref{alonolson} does not seem to be a strong estimation because of $card_p(Q)\leq p$, no better estimation has been known thus far.

It is worth noting that for $d_1=d_2=\dots=d_n=1$, Theorem \ref{olson} and Theorem \ref{alonolson} immediately follow from Corollary \ref{combnullprimes}: for $f_i(\vecc{x})=\sum_{j=1}^{m} a_{ij}x_j$, there exists a vector $\vecc{0}\neq \vecc{x}\in \{0,1\}^m$ such that $f_i(\vecc{x})\in Q_i$ for all $i$. Consequently, $J=\{j : x_j = 1 \}$ fulfills the condition (\ref{eq:cong}).

Motivated by these questions, in this paper, we give analogous theorems modulo arbitrary prime powers instead of primes, extending Corollary \ref{combnullprimes}, and give improved bounds on $F(\mathbf{d},\mathbf{Q})$.
 
%This observation also motivates our study to improve these estimations.

\subsection{Complexity aspects}

As an application of Olson's theorem, Alon, Friedland and Kalai \cite{reg} discussed the following extremal graph theoretic question. Given a prime power $p^d$ and an integer $n$, the problem is to determine the smallest value of $m$ such that for every graph on $n$ vertices and $m$ edges, there exists a nonempty $p^d$-divisible subgraph, that is, a nonempty subset of edges such that the number of edges incident to every vertex is divisible by $p^d$. Conversely, determine the maximum number of edges a graph can have without containing a nonempty $p^d$-divisible subgraph. The exact answer was given in \cite{reg}, see Theorem \ref{qdiv}.

A natural question is to determine the computational complexity of finding such a subgraph if the graph has sufficiently large number of edges. For the case $p^d=2$, the problem is equivalent to finding a cycle in a graph. In this case, there exists a polynomial time algorithm, but the problem is open in all other cases.

Due to various applications of the Combinatorial Nullstellensatz, it is also a natural question to determine the computational complexity of the corresponding search problem. An open question by West \cite{west} is about the complexity of the Combinatorial Nullstellensatz over $\mathbb{F}_2=\{0,1\}$. He conjectures that the corresponding search problem belongs to the complexity class Polynomial Parity Argument (PPA) defined by Papadimitriou \cite{papa}. This complexity class contains such computational search problems that the existence of a solution can be proved by so-called parity argument: \textit{Every finite graph has an even number of odd-degree nodes.} In this paper, we verify his conjecture.

%This class and the similar class PPAD (Polynomial Parity Argument for Directed graphs) are intensively studied nowadays.

%\pagebreak

% OUR RESULTS
\section{Main results} \label{sec:mainresults}

%For every index $i$, let us be given a polynomial  $f_i$ over $\mathbb{Z}$, and a subset $0\in Q_i\subseteq \mathbb{Z}_{p^{d_i}}$. The prime powers $p^{d_i}$ may also be different for different indices $i$. We will show in Theorem \ref{maintheorem} that with $m > \sum_{i=1}^n \deg(f_i)\cdot price(\mathbb{Z}_{p^{d_i}} \backslash Q_i)$ the  statement similar to Corollary \ref{combnullprimes} is also true. Here $price(B)$ is the minimum cost of covers of the set $B$ with some special subsets of $\mathbb{Z}_{p^d}$. These special subsets are defined by integer-valued polynomials and their prices are the degree of polynomials. Here, we note that, in the case of primes, $price(B)=|B|$ is true, so our result agrees Corollary \ref{combnullprimes}, the original estimations via Combinatorial Nullstellensatz.

%As an application, we apply these techniques to prove some new results in a generalization of Olson's theorem which was originally proposed by Alon, Friedland and Kalai in \cite{reg}. Our results improves estimations which were previously known.

%Let us define the special subsets and their price which were mentioned in the previous section. These subsets correspond to integer-valued polynomials.

Now we present the first main result of this paper: the extension of Corollary \ref{combnullprimes} for arbitrary prime powers. This theorem also implies Theorem \ref{olson} and Theorem \ref{alonolson}.

%We will show that it would extends both Corollary \ref{combnullprimes} and Olson's theroem and the generalization by Alon et al.

\begin{de} \label{covering}
Let $h(x)$ be an integer-valued polynomial in $\mathbb{Q}[x]$ such that $h(0)$ is not divisible by $p$. We say that $B \subseteq \mathbb{Z}_{p^{d}}$ is covered by a set of such integer-valued polynomials $\mathcal{H}$ if for every $b\in B$, we have $p\mid h(b)$ for at least one $h\in\mathcal{H}$. The price of the set $B$ is defined as
\[
price(B)=\min \{ \sum_{h\in\mathcal{H}} \deg(h) : B \textnormal{ is covered by }\mathcal{H} \textnormal{, such that for all } h\in\mathcal{H}, p \nmid h(0)\}.
\]
\end{de}

\begin{theo} \label{maintheorem}
Suppose that there are given some $m$-variable polynomials $f_1,f_2,\ldots,f_n$ over $\mathbb{Z}$ without constant terms and some sets $Q_1,Q_2,\ldots,Q_n$ such that $Q_i \subseteq \mathbb{Z}_{p^{d_i}}$ and $0\in Q_i$ for all $i$. If
\[
m > \sum_{i=1}^n \deg(f_i)\cdot price(\mathbb{Z}_{p^{d_i}} \backslash Q_i)
\]
then exists a $\vecc{0}\neq \vecc{x}\in \{0,1\}^m$ such that $f_i(\vecc{x}) \equiv q_i \qquad \modu{p^{d_i}} \quad \textnormal{ for some }q_i \in Q_i \textnormal{ for all }i.$
\end{theo}

We will prove this theorem in Section \ref{sec:proof}. It is easy to check that Theorem \ref{maintheorem} implies Corollary \ref{combnullprimes}:  let $d=1$, so $0\in Q \subseteq \mathbb{F}_p$. Then, $\{h(x)=x-q : q \not\in Q \}$ covers $\mathbb{F}_{p} \backslash Q$ with price $|\mathbb{F}_{p} \backslash Q|$.

Theorem \ref{olson} and Theorem \ref{alonolson} will follow from Theorem \ref{maintheorem} via the following general estimation for $F(\mathbf{d},\mathbf{Q})$ which we will prove in Section \ref{sec:olson}. 

\begin{theo} \label{genolsontheorem}
$F(\mathbf{d},\mathbf{Q})\leq \sum_{i=1}^{n} price( \mathbb{Z}_{p^{d_i}} \backslash Q_i)$.
\end{theo}

%We will show that in this case it is in the complexity class Polynomial Parity Argument (PPA) defined by Papadimitriou.  We will proved that Combinatorial Nullstellensatz over $\mathbb{F}_2$ also belongs to PPA, this result answer an open question by West \cite{west} and it is the base of the proof for PPA-membership of $2^d$-divisible subgraph problem.

In Section \ref{sec:kappa}, we will give a general and constructive bound for $price(B)$, which gives a strictly stronger estimation for $F(\mathbf{d},\mathbf{Q})$ than that one in Theorem \ref{alonolson}. We also show a wide class where this estimation is tight.

In the rest of the paper, we analyze the related computational questions. In Section \ref{sec:subgraphs}, we will prove that the $2^d$-divisible subgraph problem belongs to the complexity class Polynomial Parity Argument (PPA). We reduce the $2^d$-divisible subgraph problem to the search problem of the Combinatorial Nullstellensatz over $\mathbb{F}_2$ and in Section \ref{sec:combnull}, we verify  West's conjecture: the search problem of the Combinatorial Nullstellensatz over $\mathbb{F}_2$ is also in PPA, if the polynomial is given in a general form such as in most of the applications. In Section \ref{sec:louigi}, we focus on degree-constrained subgraphs modulo prime powers, and  we will prove an analogous theorem for Shirazi-Verstra\" ete theorem \cite{louigi}.

\pagebreak

% \subsection{Overview of this paper}

% The rest of this paper is organized as follows. In Section 2, we present a general estimation for the price of a set. It is needed to prove that the Theorem \ref{maintheorem} also implies the generalization of Olson's theorem by Alon et al. In Section 3, we present the consequences on generalization of Olson's theorem, an estimation for $F(\mathbf{d},\mathbf{Q})$. In Section 4, we present our algebraic technique to handle prime powers modulo $p$ and prove the Theorem \ref{maintheorem}. In Section 5, we study Combinatorial Nullstellensatz over $\mathbb{F}_2$. In Section 6, we analyze the $p^d$-divisible subgraph problem. and in Section 7, we focus on degree-constrained subgraphs modulo prime powers.

\section{The proof of Theorem \ref{maintheorem}}
\label{sec:proof}

The proof of Theorem \ref{maintheorem} presented here is similar to the proof of Theorem \ref{alonolson} in \cite{reg}. Alon et al. used a similar polynomial to the one in Equation (\ref{eq:maintheorem}), %in the proof of  Theorem \ref{maintheorem}
however, they used only the special construction of Equation (\ref{eq:alonconstr}) instead of arbitrary polynomials.
%, see Section \ref{sec:olson}.
Now we extend it to arbitrary integer-valued polynomials $h$ and we can use more than one polynomial at the same time.

We apply the following corollary of Gregory-Newton formula for integer-valued polynomials. \cite{integervalued}

\begin{theo}
Let $h(x)$ be an integer-valued polynomial in $\mathbb{Q}[x]$, namely, for every integer $T$, $h(T)$ is an integer. Then, $h(x)$ can be written as $\sum_{r=0}^d \alpha_r {x \choose r}$ where $\alpha_r \in \mathbb{Z}$.
\end{theo}

In the following abstract definitions, one can think of the polynomial $f$ as the 'true meaning' of the problem such as $f_i$ in Corollary \ref{combnullprimes}, and 
one can think of the polynomial $h$ as the covering polynomial in Definition \ref{covering}.

The key idea is in the following observation. Although, $h(x)$ may have non-integral coefficients, we can construct a polynomial $\Psi^h(f)$ over $\mathbb{Z}$ that satisfies the equality $\Psi^h(f)(\vecc{s}) = h(f(\vecc{s}))$, if $\vecc{s}=(s_1,s_2,\dots,s_m) \in \{0,1\}^m$.
Since $\Psi^h(f)$ have integral coefficients, it can be considered over $\mathbb{F}_p$, so some information over $\mathbb{Z}$ -- and hence, modulo $p^d$ -- can be encoded over $\mathbb{F}_p$. 

%one can think of the polynomial $f$ as the real meaning of the problem like $\sum_{j=1}^m a_{ij}x_j$ in the Olson-type problems, and one can consider the polynomial $h$ which helps to encode more information over $\mathbb{F}_p$ than only the residues of $f(\vecc{x})$ modulo $p$.
 
\begin{de}
Let $f=\sum_{i=1}^{k} p_i$ be a polynomial over $\mathbb{Z}$, where each $p_i$ is a monomial with coefficient 1.
% in the form $\prod_{j=1}^{n} x_j^{t_{ij}}$, where each $t_{ij}$ is a nonnegative integer. 
Let $\Psi_r(f)\in \mathbb{Z}[x_1,\dots,x_m]$ be the following polynomial:
\[
\Psi_r(f) = \sum_{1\leq i_1<i_2<\dots<i_r\leq k} p_{i_1} p_{i_2} \dots p_{i_r}
\]
\end{de}

The degree of the constructed polynomial $\Psi_r(f)$ is at most $r \cdot \deg(f)$.
It is worth noting that if $\vecc{s}=(s_1,s_2,\dots,s_m) \in \{0,1\}^m$, then the possible values of $p_i(\vecc{s})$ are 0 and 1, so if the number of $p_i$s such that $p_i(\vecc{s})=1$ is $c$, then the number of terms in $\Psi_r(f)$ that are 1 at $\vecc{s}$ is precisely ${c \choose r}$.

\begin{de}
Let $f=\sum_{i=1}^{k} p_i$ be a polynomial over $\mathbb{Z}$ as above. Let $h(x)=\sum_{r=0}^d \alpha_r {x \choose r}$ be an integer-valued polynomial in $\mathbb{Q}[x]$ where $\alpha_r \in \mathbb{Z}$.
Let $\Psi^h(f)\in \mathbb{Z}[x_1,\dots,x_m]$ be the following polynomial:
\[
\Psi^h(f) = \sum_{r=0}^d \alpha_r \Psi_r(f)
\]
\end{de}

Note that, $\deg(\Psi^h(f)) \leq deg(h) \cdot \deg(f)$. In the following lemma, we can obtain the benefit of these definitions: $h(f(\vecc{x}))$ can be written as a polynomial with integer coefficients.

\begin{lemma} \label{psi}
Let $f=\sum_{i=1}^{k} p_i$ be a polynomial over $\mathbb{Z}$ as above.  Let $h(x)=\sum_{r=0}^d \alpha_r {x \choose r}$ be an integer-valued polynomial as above.
Further, let $\vecc{s}=(s_1,s_2,\dots,s_m) \in \{0,1\}^m$.
% and let $c = f(\vecc{s}) = \sum_{i=1}^{k} \prod_{j=1}^{n} s_j^{t_{ij}}$.

Then,
\[
\Psi^h(f)(\vecc{s}) = h(f(\vecc{s}))
\]
\end{lemma}

\begin{proof}
Let $c = f(\vecc{s})$. Then, the number of terms in $\Psi_r(f)$ that are 1 at $\vecc{s}$ is precisely ${c \choose r}$, the other terms are 0. So 
\[
\Psi^h(f)(\vecc{s}) = \sum_{r=0}^d \alpha_r \Psi_r(f)(\vecc{s}) = \sum_{r=0}^d \alpha_r {c \choose r} = h(c) = h(f(\vecc{s})). \qquad
\]
\end{proof}

Now we are ready to prove Theorem \ref{maintheorem}.
%It is worth noting that the proof is similar to the proof of Theorem \ref{alonolson} in \cite{reg}. 

\begin{proof}[Proof of Theorem \ref{maintheorem}]
To simplify the notation, let $C_i$ be the complementary set of $Q_i$, that is, $\mathbb{Z}_{p^{d_i}} \backslash Q_i$. Due to the definitions, there exists a set of polynomials $\mathcal{H}_i$ which covers $C_i$ with the total degree $price(C_i)$.
Let us consider the following polynomial in $\mathbb{F}_p[x_1,\dots,x_m]$:
\[
f(\vecc{x}) = \prod_{i=1}^n \prod_{ h \in \mathcal{H}_i} \Psi^h\left( f_i(\vecc{x}) \right) - c \cdot \prod_{j=1}^m (1-x_j), \label{eq:maintheorem} \tag{1}
\]

where $c$ is a nonzero constant to be defined later.

The degree of the first part of the polynomial is $\sum_{i=1}^n \left(\deg(f_i) \cdot \sum_{h \in \mathcal{H}_i} deg(h) \right) = \sum_{i=1}^n \deg(f_i) \cdot price(C_i) < m$, so the degree of the polynomial $f$ is $m$, and the coefficient of $x_1x_2\dots x_m$ is $-c \cdot (-1)^m \neq 0$.

If $\vecc{x} = \vecc{0}$, then the first part is nonzero, because $h(0)$ is not divisible by $p$ for every $h\in \mathcal{H}_i$. Let $c$ be the value of the first part at $\vecc{0}$. So, $f(\vecc{0})=c-c=0$. Let $t_j = 1, S_j = \{0,1\}$. Then, the conditions of Combinatorial  Nullstellensatz hold, so there exists an $\vecc{s}\in \{0,1\}^m$ such that $f(\vecc{s})\neq 0$. For this $\vecc{s}\in \{0,1\}^m$, at least one component of $\vecc{s}$ is 1 due to $f(\vecc{0})=0$, so the second part of the polynomial $f$ is zero, and hence, the first part must be nonzero at vector $\vecc{s}$. This means that $f_i(\vecc{s})$ is not covered by any $h\in \mathcal{H}_i$, so it is not $C_i$, hence it must be in $Q_i$. So, $f_i(\vecc{s}) \equiv q_i \qquad \modu{p^{d_i}} \quad \textnormal{ for some }q_i \in Q_i \textnormal{ for every }i=1,\dots ,n$, completing the proof.
\end{proof}

\section{The generalization of Olson's theorem: estimation for $F(\mathbf{d},\mathbf{Q})$} \label{sec:olson}

Let us now derive Theorem \ref{genolsontheorem} from Theorem \ref{maintheorem}.

%Before the proof, let us remark why $0\in Q_i$ is necessary requirement in the definition $F(\mathbf{d},\mathbf{Q})$. If $0\in Q_i$ is not required,  $F(\mathbf{d},\mathbf{Q})$ could be infinite: if all the $a_{ij}$'s equal zero, there never exists a nontrivial subset that fulfills condition (\ref{eq:cong}).

%The following theorem immediately follows from Theorem \ref{maintheorem}. This result gives stronger estimation for $F(\mathbf{d},\mathbf{Q})$ than results which were previously known.

%\begin{theo} \label{genolsontheorem}
%$F(\mathbf{d},\mathbf{Q})\leq \sum_{i=1}^{n} price( \mathbb{Z}_{p^{d_i}} \backslash Q_i)$.
%\end{theo}

\begin{proof}[Proof of Theorem \ref{genolsontheorem}]
Let $f_i(x_1,\ldots,x_m)=\sum_{j=1} a_{ij}x_j$ and $m > \sum_{i=1}^{n} price( \mathbb{Z}_{p^{d_i}} \backslash Q_i)$. Applying Theorem \ref{maintheorem}, there exists a vector $\vecc{0}\neq \vecc{x}\in \{0,1\}^m$ such that $f_i(\vecc{x}) \equiv q_i \qquad \modu{p^{d_i}} \quad \textnormal{ for some }q_i \in Q_i \textnormal{ for all }i$.

Let $J= \{j : x_j=1 \}$. Then,
\[
\sum_{j\in J} a_{ij} = f_i(\vecc{x}) \equiv q_i \qquad \modu{p^{d_i}} \quad \textnormal{ for some }q_i \in Q_i \textnormal{ for every }i=1,\dots ,n.
\]

Hence, $F(\mathbf{d},\mathbf{Q})\leq \sum_{i=1}^{n} price( \mathbb{Z}_{p^{d_i}} \backslash Q_i)$.
\end{proof}

We are ready to show that Theorem \ref{genolsontheorem} implies Theorem \ref{alonolson} and its special case, Theorem \ref{olson}.

% Let $0 \in Q\subseteq \mathbb{Z}_{p^{d}} $ be a set of distinct integers modulo $p$, and $B=\mathbb{Z}_{p^{d}} \backslash Q$. Then $k_{d-1}=p-1$ and $B_{d-1}$ is the complement such a set that contains distinct integers modulo $p$. Then, $k_{d-2}=p-1$, and so on. At last, $k_0 = p - |Q|$. So $\kappa(B)=p^d-|Q|=|B|$.

% $\kappa$ is a monotone function, so for arbitrary set $0 \in Q\subseteq \mathbb{Z}_{p^{d}}$, $\kappa(\mathbb{Z}_{p^{d}} \backslash Q)\leq p^d - card_p(Q)$ follows, where $card_p(Q)$ denotes the number of distinct elements in $Q$ modulo $p$.

% $price(\mathbb{Z}_{p^{d}} \backslash Q) \leq \kappa(\mathbb{Z}_{p^{d}} \backslash Q) \leq p^d - card_p(Q)$.

Let $d$ be arbitrary, and $0\in Q' \subseteq \mathbb{Z}_{p^d}$ be a set of distinct integers modulo $p$. Then, let
\[
h(T):=\frac{1}{p^\delta}\prod_{q\not \in Q'}( T-q)\textnormal{, where }\delta = \sum_{r=0}^{d-1} (p^r-1). \label{eq:alonconstr} \tag{2}
\]

For every integer $T$, in the product $\prod_{q\not \in Q'}( T-q)$, at least $p^{d-r}-1$ numbers are divisible by $p^r$ for every $1\leq r \leq d$. Hence $h(T)$ is an integer-valued polynomial. Further, $h(T)$ is not divisible by $p$ if and only if no factor is divisible by $p^{d}$. So $h(T)$ is divisible by $p$ if and only if $T\equiv q\quad \modu{p^d}$ for some $q\not \in Q'$.
Hence, $h(0)$ is not divisible by $p$ and $h(T)$ covers $\mathbb{Z}_{p^d} \backslash Q'$ with price $deg(h)=p^d-|Q'|$.

This implies that if $0\in Q$ is an arbitrary subset of $\mathbb{Z}_{p^d}$, $price(\mathbb{Z}_{p^d} \backslash Q) \leq p^d -  card_p(Q) $, so Theorem \ref{alonolson} follows from Theorem \ref{genolsontheorem}.

\label{sec:kappa}
%To make real significance for the Theorem \ref{maintheorem}, the major task is to give strong estimations for prices of arbitrary subsets. In this paper, we can give the following estimation for $price(B)$. In the next section, we present a large class of special cases where this condition is tight.

%Although, we easily prove that Theorem \ref{maintheorem} implies all the theorems in this area, it is a major task to give strong and algorithmically computable estimations for prices of arbitrary subsets. Here, we present such an estimation.

Furthermore, Theorem \ref{genolsontheorem} enables to obtain strictly stronger bounds than the one in Theorem \ref{alonolson} via the following general constructive estimation on $price(B)$.

\begin{de} \label{kappa}
For set $B$ of integers modulo $p$, let $\kappa(B)=|B|$. For any $d>1$ and for set $B$ of integers modulo $p^d$, let us define $k$ as the cardinality of the set $\{b\in B : b\textnormal{ is divisible by }p^{d-1}\}$ and $\hat{B}$ as the set of such residues modulo $p^{d-1}$ that appear in $B$ more than $k$ times. Then, let $\kappa(B)=k\cdot p^{d-1}+\kappa(\hat{B})$.
\end{de}

The following definition for $\kappa$ is equivalent to the above one. It gives a way to compute the value of $\kappa(B)$. Let $B$ be a set of integers modulo $p^d$. Now, we define integers $k_{d-1},\dots,k_{0}$ and sets $B_{d},\dots,B_{1}$ such that $B_r$ is a set of integers modulo $p^r$. Let $B_{d}=B$ and $k_{d-1}=|\{b\in B : b\textnormal{ is divisible by }p^{d-1}\}|$. Then, for $r=d-1,\ldots,2, 1$, if $B_{r+1}$ is given, let $B_{r}$ be the set of such residues modulo $p^r$ that appear in $B_{r+1}$ more than $k_{r}$ times and let $k_{r-1} = |\{b\in B_{r} : b\textnormal{ is divisible by }p^{r-1}\}|$. Then, $\kappa(B)=\sum_{r=0}^{d-1} k_r\cdot p^r$.

\begin{example}
Let $p^d=5^3=125$, $B=\{1,2,5,6,13,20,40,42,50,51,52,56,69,70,87,95,100,101,102,112\}$.
Then, $k_2=2$, because two integers in $B$ are divisible by 25 (50 and 100). Then, $B_{2}=\{1,2,12,20\}$. For instance, $20 \in B_2$, because $20\equiv 70 \equiv 95 $ are in $B_3$, but $6 \not\in B_2$, because only $6\equiv 56$ are in $B_3$. So on, $k_1=1$, $B_{1}=\{2\}$, and $k_0=1$. Hence, $\kappa(B)=2\cdot 25 + 1\cdot 5 + 1 \cdot 1 = 56$.
\end{example}

\begin{theo} \label{kappabound}
With the above definition, $price(B) \leq \kappa(B)$ holds.
\end{theo}

\begin{proof}
We construct a polynomial that covers a complete $p^r$-residue system modulo $p^{r+1}$ with price $p^r$. 

Let $q_1,q_2,\dots,q_{p^r}$ be a complete $p^r$-residue system and let
\[
h(T):=\frac{1}{p^\delta}\prod_{i=1}^{p^r}( T-q_i)\textnormal{, where }\delta = \sum_{j=0}^{r-1} p^{j}.
\]

For every integer $T$, the integers $T-q_1,T-q_2,\dots,T-q_{p^r}$ also form a complete residue system modulo $p^r$, so in the product $\prod_{i_=1}^{p^r}( T-q_i)$, $p^{r-j}$ factors are divisible by $p^j$ for every $1\leq j \leq r$. Hence, the product is divisible by $p^\delta$ and $h(T)$ is an integer-valued polynomial.

Further, $h(T)$ is divisible by $p$ if and only if the factor which is divisible by $p^r$ is also divisible by $p^{r+1}$. This means $T \equiv q_i \quad \modu{p^{r+1}}$ for some $i$, that is, $h(T)$ covers $q_1,q_2,\dots,q_{p^r}$ modulo $p^{r+1}$, precisely, it covers the set $\{q \in \mathbb{Z}_{p^d} : q \equiv q_i \quad \modu{p^{r+1}}\textnormal{ for some }i \}$ with price $p^r$. 

Then, by Definition \ref{kappa}, the statement immediately follows: one can cover the integers that are divisible by $p^{d-1}$ with $k$ such conditions. These conditions also covers other residues $k$ times, so such residues are not covered by the conditions that appear more than $k$ times. These remaining residues are in $\hat{B}$ modulo $p^{d-1}$ and they can be covered with $\kappa(\hat{B})$.
\end{proof}

\subsection{A special case when Theorem \ref{genolsontheorem} is tight}

Here we show a special case when the theorem is tight. This statement shows a wide class where Theorem \ref{genolsontheorem} and hence Theorem \ref{maintheorem} give tight estimation. In general, tightness is not yet known. This result also shows cases when Theorem \ref{alonolson} gives strictly weaker estimation than the one in  Theorem \ref{genolsontheorem}.

\begin{de} \label{rzeroset}
Let $R$ be a subset of $\{ 0,1,\dots,d-1\}$. Let us define the set $\Omega \subseteq \mathbb{Z}_{p^{d}}$ by the following property: $c \in \Omega$ if and only if $c^{(r)}=0$ for every $r\in R$ in the $c^{(d-1)}\dots c^{(1)}c^{(0)}$ form of $c$ in base $p$. We call $\Omega$ the $R$-zero set modulo $p^d$. Let $\sigma(R)= (p-1)\sum_{r \in R} p^r$.
\end{de}

\begin{theo} \label{rzerolson}
Let $R_i$ be an arbitrary subset of $\{ 0,1,\dots,d_i-1\}$ and denote the $R_i$-zero set modulo $p^{d_i}$ by  $\Omega_i$. Then, $F(\mathbf{d},\mathbf{\Omega}) = \sum_{i=1}^{n} \sigma(R_i)$.
\end{theo}

\begin{proof}
We show that $\sigma(R)=\kappa(\mathbb{Z}_{p^{d}} \backslash \Omega)$ and hence $F(\mathbf{d},\mathbf{\Omega}) \leq \sum_{i=1}^{n} \sigma(R_i)$ by Theorem \ref{genolsontheorem} and Theorem \ref{kappabound}.

The proof is by induction on $d$. Let $B = \mathbb{Z}_{p^{d}} \backslash \Omega$. Further, let $d'=d-1$ and let $\Omega'$ be the $R'=R\backslash \{d-1\}$-zero set modulo $p^{d'}$ and $B' = \mathbb{Z}_{p^{d'}} \backslash \Omega'$. By induction, $\sigma(R')=\kappa(B')$.

If $d-1 \not\in R$, then $k=|\{b\in B : b\textnormal{ is divisible by }p^{d-1}\}|=0$, and $\hat{B}=B'$. Then $\kappa(B)=0 + \kappa(B')=\sigma(R')=\sigma(R)$.

If $d-1 \in R$, then $k=|\{b\in B : b\textnormal{ is divisible by }p^{d-1}\}|=p-1$ and $\hat{B}=B'$. Then $\kappa(B)=(p-1)\cdot p^{d-1} + \kappa(B')=(p-1)\cdot p^{d-1}+ \sigma(R')=\sigma(R)$. 

Moreover, these bounds are tight: $F(\mathbf{d},\mathbf{\Omega}) = \sum_{i=1}^{n} \sigma(R_i)$, because if $m=\sum_{i=1}^{n} \sigma(R_i)$, then there exists integers $a_{ij}$ such that the proper nontrivial subset does not exist. Let $a_{ij}$ be $-1$ where $\sum_{l=1}^{i-1} \sigma(R_l) < j \leq \sum_{l=1}^{i} \sigma(R_l)$, and zero otherwise.
However, in the range $-p^{d_i},\dots,0$, the largest integer of the $R_i$-zero sets modulo $p^{d_i}$ is
\[
\sum_{r \not \in R_i} (p-1)\cdot p^r = p^{d_i} - 1 -  \sum_{r \in R_i} (p-1)\cdot p^r = p^{d_i} - 1 - \sigma(R_i) \equiv -\sigma(R_i) - 1 \quad \modu{p^{d_i}}
\]
and $-\kappa{R_i} \leq \sum_{j\in J} a_{ij} \leq 0$, hence, no nonempty subset exists that fulfills the condition (\ref{eq:cong}).
\end{proof}

\section{Complexity aspects of the Combinatorial Nullstellensatz}
\label{sec:combnull}
Due to various applications of the Combinatorial Nullstellensatz, it is a natural and important question to determine the computational complexity of the corresponding search problem. Now, we study the complexity of the Combinatorial Nullstellensatz over $\mathbb{F}_2$.
It is worth noting that if $t_i=0$ for some indices, then we could choose $|S_i|=1$, so in an appropriate vector $(s_1,s_2,\dots, s_m)$, we have to choose the only element of $S_i$ to $s_i$ and therefore, we could replace $x_i$ by the only element of $S_i$ in $f$. Hence, we may assume that $S_i=\mathbb{F}_2$ for every index $i$ and the problem is finding a vector $(s_1,s_2,\dots, s_m) \in \mathbb{F}_2^m$ such that $f(s_1,s_2,\dots s_m)\neq 0$.

The complexity of finding such a vector whose existence is guaranteed by the Combinatorial Nullstellensatz depends on the input form of the given polynomial.

It is easy to check that the problem belongs to P if the polynomial is given explicitly as the sum of monomials. First, we can replace the term $x_{i_1}^{t_{i_1}}x_{i_2}^{t_{i_2}}\dots x_{i_k}^{t_{i_k}}$ by $x_{i_1}x_{i_2}\dots x_{i_k}$, because these are equal due to the fact $0^t=0, 1^t=1$ in $\mathbb{F}_2$. Substitute 0 and 1 to $x_1$: let $g(x_2,\dots,x_{n})=f(0,x_2,\dots,x_{n})$ and $h(x_2,\dots,x_{n})=f(1,x_2,\dots,x_{n})$. If in $f$ the coefficient of $x_2x_3\dots x_m$ is nonzero, then in $g$ the coefficient of $x_2x_3\dots x_m$ will be also nonzero. If it is zero, in $h$ the coefficient of $x_2x_3\dots x_m$ will be nonzero. Then, substitute 0 and 1 to $x_2$ and in one of them the coefficient of $x_3x_4\dots x_m$ will be nonzero, and so on. Finally, we obtain a constant nonzero polynomial, and this means that for this substitution  $\vecc{s} \in \mathbb{F}_2^m$, $f(\vecc{s}) \neq 0$ holds. It is worth noting that a similar polynomial time algorithm can be obtained over arbitrary finite field, if the polynomial is given explicitly.

However, if the polynomial is given as the sum of products of polynomials (such as in most of the applications), the problem is not known to be solvable in polynomial time. An open question in \cite{west} is about the complexity of the Combinatorial Nullstellensatz conjecturing that the problem over $\mathbb{F}_2$ belongs to the class Polynomial Parity Argument (PPA) defined by Papadimitriou in \cite{papa}.

In this section, we verify this conjecture: we prove that the Combinatorial Nullstellensatz over $\mathbb{F}_2$ is in PPA if the polynomial is given as the sum of products of polynomials. Consequently, the applications given in Sections \ref{sec:subgraphs} and \ref{sec:louigi} also belong to PPA.

Roughly speaking, the class PPA is a subclass of the semantic class TFNP, the set of all total search problems. A search problem is called {\em total} if the corresponding decision problem is trivial, that is, for every feasible input, there exists a solution. A total problem is usually equipped with a mathematical proof showing that it belongs to TFNP, so the problems can be classified based on their proof styles. The complexity class PPA is the class of all search problems whose totality is proved using the parity argument: \textit{Every finite graph has an even number of odd-degree nodes.}

This class PPA can be defined with a canonical complete problem, the \textsc{End Of The Line}. Hence, a computational search problem is in PPA if and only if it is reducible to the problem \textsc{End Of The Line}.

%In this problem, we are given an undirected graph with exponential number of nodes and maximum degree two. The input to the problem is not a complete list of the nodes and edges: it may be exponentially large in the size of the input. Instead, a polynomial time algorithm describes the graph and this algorithm outputs the neighbours of an input node. In the problem, we are given a node of degree one, called the standard leaf. The task is to find another node which has exactly one incident edge.

In this problem, we are given a graph $G=(V,E)$ on exponentially many nodes. It can be assumed that each node has an unique code from $\Sigma^n$, that is $V\subseteq \Sigma^n$. The edges of the graph are described by a polynomial time algorithm in $n$. This polynomial time pairing function is the following.

%The definition of the problem \textsc{Short End Of The Line} is the following. Note that the required node exists due to the parity argument: \textit{Every finite graph has an even number of odd-degree nodes.}

%\begin{center}
%\fbox{
%\parbox{0.85\linewidth}{
%\smallskip
%
%\noindent
%\center
%\problem{Short End of the Line.}
%\vspace{1mm}
%
%\begin{tabularx}{\linewidth}{lX}
%\textit{Input:}& 
%an undirected finite graph $G=(V,E)$ \textit{in which each node has at most two incident edges} in the above way. The edges of the graph is described a polynomial algorithm which lists for a node its neighbours. Furthermore, a node $\varepsilon$ is given which has exactly one incident edge. \\
%\textit{Find:}& another node which has exactly one incident edge.
%\end{tabularx}
%}}
%\end{center}
%\medskip

%Papadimitriou shows that this task is computationally equivalent to the problem in which the nodes may have more neighbours and a polynomial time pairing function is given.

For an undirected graph $G=(V,E)$, the function $\phi : V\times V \rightarrow V\cup\{*\}$ is called a {\em pairing function}, if it satisfies the following conditions: if $vw$ is not an edge of $G$, let $\phi(v,w)=*$. Otherwise, it outputs a node $w'=\phi(v,w)$ such that $w'$ is also connected to $v$ and $\phi(v,\phi(v,w))=w$ holds. Furthermore, for every $v$, at most one such node $w$ exists with property $\phi(v,w)=w$. 

It means that $\phi$ pairs up the neighbours of an input node $v$: for an even-degree node $v$, it pairs its neighbours completely, and for an odd-degree node $v$, $\phi$ pairs all but one neighbours. The task is to find an odd-degree node $v$ and a node $w$ such that $\phi(v,w)=w$. This node $w$ verifies that $v$ is an odd-degree node.

The problem \textsc{End Of The Line} can be defined as follows.

\begin{center}
\fbox{
\parbox{0.85\linewidth}{
\smallskip

\noindent
\center
\problem{End of the Line.}
\vspace{1mm}

\begin{tabularx}{\linewidth}{lX}
\textit{Input:}&
an undirected finite graph $G=(V,E)$ in the above way. The edges of the graph is described by a polynomial time pairing function.lists for a node its neighbours. Furthermore, a node $\varepsilon$ is given which has odd number of edges and a node $\delta$ which shows it: $\phi(\varepsilon,\delta)=\delta$. \\
\textit{Find:}&  another node $v$ which has odd number of edges and a node $w$ which give the certificate $\phi(v,w)=w$. 
\end{tabularx}
}}
\end{center}
\medskip

In order to prove problems belonging to PPA, we give reductions to the problem \textsc{End Of The Line}.

It is worth noting that this problem is computationally equivalent to the problem in which the nodes have at most two neighbours, a node of degree one is given and the task is to find another node which has exactly one incident edge. (Instead of the polynomial time pairing function, a polynomial time algorithm is given which outputs the neighbours of an input node.) It is easy to see that this is an easier problem, however, Papadimitriou showed that they are 
computationally equivalent.

In \cite{papa}, Papadimtriou shows that the following computational problem \textsc{Chévalley MOD 2} belongs to the class PPA. The required vector exists due to Chévalley's following theorem.

\begin{theo}[Chévalley]
Let $\mathbb{F}$ be a finite field with characteristic $p$. Let $p_1,p_2,\dots, p_n$ be polynomials in $m$ variables over $\mathbb{F}$. Suppose that $\sum_{i=1}^{n} \deg(p_i) < m$. Then, the number of common solutions of the polynomial equation system $p_i(x_1,\dots,x_m)=0$ $(i=1\dots n)$ is divisible by $p$. In particular, if there is a solution, there exists another.
\end{theo}

\begin{center}
\fbox{
\parbox{0.85\linewidth}{
\smallskip

\noindent
\center
\problem{Chévalley MOD 2.}
\vspace{1mm}

\begin{tabularx}{\linewidth}{lX}
\textit{Input:}& 
polynomials $p_1,p_2, \dots , p_n$  in $\mathbb{F}_2[x_1,\dots x_m]$ such that $\sum_{i=1}^n \deg(p_i) < m$. Also, we are given a root $(c_1,c_2,\dots,c_m) \in \mathbb{F}_2^m$  of the equation system $p_i(\vecc{x})=0$ ($i=1,\dots,n$)\\
\textit{Find:}& another root of the equation system $p_i(\vecc{x})=0$ ($i=1,\dots,n$).
\end{tabularx}
}}
\end{center}
\medskip

Using Theorem \ref{combulltheo} and that Chévalley's theorem can be proved via a reduction to the Combinatorial Nullstellensatz, see \cite{alon}, one can give an alternative proof for the PPA membership of \textsc{Chévalley MOD 2}. Originally, this reduction motivates West's question \cite{west} about the  complexity of the Combinatorial Nullstellensatz.

Now, let us define the following computational problem. Note that the required vector exists due to the Combinatorial Nullstellensatz.

\begin{center}
\fbox{
\parbox{0.85\linewidth}{
\smallskip

\noindent
\center
\problem{Combinatorial Nullstellensatz over $\mathbb{F}_2$.}
\vspace{1mm}

\begin{tabularx}{\linewidth}{lX}
\textit{Input:}& 
a polynomial $f$ in $m$ variables in a general form $f= \sum_{i=1}^{k} \left( \prod_{j=1}^{m_i} p_{ij} \right)$, where $p_{ij}$ is an explicitly given polynomial in $\mathbb{F}_2[x_1,\dots x_m]$, $k,m_i$ and the number of monomials of $p_{ij}$ is polynomially bounded in $m$. Suppose that $\sum_{j=1}^{m_i} \deg(p_{ij}) \leq m$ for all $i$ and there is a polynomial time pairing function which can pair up all but one terms $x_1x_2\dots x_m$ to prove that the degree of $f$ is $m$ and the coefficient of $x_1x_2\dots x_m$ is nonzero.
\\
\textit{Find:}& a vector $(s_1,s_2,\dots, s_m) \in \mathbb{F}_2^m$ such that $f(s_1,s_2,\dots,s_m) \neq 0$.
\end{tabularx}
}}
\end{center}
\medskip

\begin{theo} \label{combulltheo}
The \textsc{Combinatorial Nullstellensatz over $\mathbb{F}_2$} is polynomially reducible to \textsc{End of the Line}.
Consequently, The \textsc{Combinatorial Nullstellensatz over $\mathbb{F}_2$} is in PPA.
\end{theo}

The proof of Theorem \ref{combulltheo} is similar to the proof for the PPA membership of \textsc{Chévalley MOD 2}. Our construction is based on that proof, nevertheless, we need a new key idea about the upper-level pairing function which pairs up blocks whose value at 1 for the substitution $\vecc{x}$.
 
We construct a graph, the nodes correspond to the vectors and the terms. The nodes with odd degree correspond to the vectors $\vecc{x}$ such that $f(\vecc{x})\neq 0$ and an extra node $w$. As we mentioned, we have to present a pairing function. It can be done easily at the terms, but it is more complicated at the vectors. The main idea here is the following.

We call the polynomials $\prod_{j=1}^{m_i} p_{ij}$ as the blocks of the input polynomial $f$. Each term is the product of monomials from the given polynomials of a block, so each term in the $i$th block can be represented by an $(m_i+1)$-tuple of integers: $(i,a_{i,1},\dots, a_{i,m_i})$. The first coordinate shows the block the term belongs to, and the other coordinates show the monomials the term is product of: it is the product of $a_{i,j}$th monomials of $p_{ij}$. (Note that the same term might have more than one occurrence and these occurrences are represented by different tuples.) In the next proof, we will pair up these tuples.

\begin{remark}
In a standard PPA-type problem definition it is required that the assumptions of the problem should be in NP. If the input is feasible, we have to return a solution, but if the input is infeasible, we have to return a polynomial certificate of infeasibility.

It is easy to check that the assumptions in the definition of the \textsc{Combinatorial Nullstellensatz over $\mathbb{F}_2$} are in NP. In the case of an infeasible input, we can give the following certificate: the index $i$ such that $\sum_{j=1}^{m_i} \deg(p_{ij}) > n$ or two occurrences of the term $x_1x_2\dots x_m$ for which the polynomial time pairing function fails.

%The given pairing function in the definition of \textsc{Combinatorial Nullstellensatz MOD 2} is an ordinary *thing* in PPA-like problem definitions. It is worth 
\end{remark}

\begin{proof}[Proof of Theorem \ref{combulltheo}]

We shall construct a graph $\Gamma$ whose odd-degree nodes precisely correspond to appropriate vectors $\vecc{s}$ such that $f(\vecc{s}) \neq 0$ and furthermore, we add an extra node $w$: the standard leaf.

The graph is bipartite. The nodes on one side are all the vectors in $\mathbb{F}_2^m$ and the extra node $w$. The nodes of the other side are the terms of the polynomial $f= \sum_{i=1}^{k} \left( \prod_{j=1}^{m_i} p_{ij} \right)$. Each term is represented in the above way as an $(m_i+1)$-tuple of integers.

There is an edge between vector $\vecc{x}$ and term $t$ if and only if $t(\vecc{x})=1$, and there is an edge between the extra node $w$ and the term $t$ if and only if $t(\vecc{x})=x_1x_2\dots x_m$.

It is easy to see that for a vector $\vecc{x}$, $f(\vecc{x})\neq 0$ holds if and only if its degree is odd. The extra node $w$ also has an odd degree because the coefficient of $x_1x_2\dots x_m$ is nonzero due to the assumptions.

All nodes in the other side have even degree. In $\Gamma$, the degree of each term $t(\vecc{x})=x_1x_2\dots x_m$ is precisely 2, because it is connected only to the vector $(1,1,\dots,1)$ and to the extra $w$ node. Let $t$ be any other term, and let $x_l$ be a variable not appearing in $t$. Then, if $t$ is connected to $(s_1,s_2,\dots,0,\dots,s_m)$, it is also connected to $(s_1,s_2,\dots,1,\dots,s_m)$, so the degree of these nodes are even.

Therefore, odd-degree nodes are precisely the vectors  $\vecc{s}$ such that $f(\vecc{s}) \neq 0$ and the extra node $w$.

However, the nodes of this graph have exponentially large degrees, and therefore we must exhibit a pairing function between the edges incident to a node.

For a node corresponding to the term $t(\vecc{x}) \neq x_1x_2\dots x_m$, we pair up the vector $\vecc{x}$ for which $t(\vecc{x})=1$ to $(x_1,x_2,\dots,1-x_l,\dots x_m)$ where $x_l$ is such a variable which does not appear in $t$. (We choose the smallest such index $l$.) The degree of nodes corresponding to terms $x_1x_2\dots x_m$ in this side is only 2, its edges can be simply paired up.

For such node corresponding to a vector $\vecc{x}$ that $f(\vecc{x}) = 0$ holds, we should pair up the terms such that $t(\vecc{x})=1$. Suppose that the term $t$ is represented by $(i,a_{i1},\dots, a_{ij}, \dots, a_{i,m_i})$.

Denote its block by $g=\prod_{j=1}^{m_i} p_{ij}$. If $g(\vecc{x})=0$, then there is an index $j$ such that $p_{ij}(\vecc{x})=0$. Pick the smallest such $j$. There is an even number of monomials of $p_{ij}$ such that $p_{ij}(\vecc{x})=1$. We pair these monomials by a pairing function $\phi_i$. Then the mate of term $(i,a_{i1},\dots, a_{ij}, \dots, a_{i,m_i})$ is
$(i,a_{i1},\dots, \phi_i(a_{ij}),\dots, a_{i,m_i})$.

It is a more complicated case when $g(\vecc{x})=1$. Since $f(\vecc{x})=0$, there is an even number of indices $l$, such that $\left( \prod_{j=1}^{m_l} p_{lj} \right)$ is 1 at $\vecc{x}$. We pair these blocks by a pairing function $\phi$. So, for $i$ and every $j=1,\dots ,m_i$, $p_{ij}$ is 1 at $\vecc{x}$, and we can pair all but one monomials of $p_{ij}$ with $p_{ij}(\vecc{x})=1$ by a pairing function $\phi_{ij}$. One of them does not have a mate, denote its index by $\omega_{ij}$. If $a_{ij}=\omega_{ij}$ for all indices $j$, then we define its mate to be $(\phi(i),\omega_{\phi(i),1},\dots, \omega_{\phi(i),m_{\phi(i)}})$.
Otherwise there is an index $j$ such that $a_{ij} \neq \omega_{ij}$. Pick the smallest such $j$. Then the mate of $(i,a_{i1},\dots, a_{i,m_i})$ is defined as $(i,a_{i1},\dots,\phi_{ij}(a_{ij}),\dots, a_{i,m_i})$.

Observe that this gives a bijection and a correct pairing function.

For such node corresponding to a vector $\vecc{x}$ that $f(\vecc{x}) = 1$ holds, we should pair up all but one terms such that $t(\vecc{x})=1$. If $t$ is a term of such block $g=\prod_{j=1}^{m_i} p_{ij}$ that $g(\vecc{x})=0$ holds, it can be paired up similarly to the previous case. If $g(\vecc{x})=1$, we pair these blocks by a pairing function. One of them does not have a mate, denote its index by $\Omega$. If $g$ is not the block with index $\Omega$, the pairing can be similar to the previous case. If $g$ is the block with index $\Omega$, we can pair up the terms similarly to the previous case, only the term $t$ represented by $(\Omega,\omega_{\Omega 1},\dots, \omega_{\Omega,m_\Omega})$ does not have a mate. So we paired up all but one neighbours of the node corresponding to the vector $\vecc{x}$.

Finally, we pair up the terms which are connected to the extra node $w$. These are the terms $x_1x_2\dots x_m$. Due to the assumptions, there is a polynomial time pairing function which can pair up all but one terms $x_1x_2\dots x_m$, so it can pair up the nodes which are connected to the extra node $w$.

We presented a polynomial algorithm that computes the mate of an edge out of a node, so the proof is complete. \end{proof}

\section{Complexity of finding divisible subgraphs}
\label{sec:subgraphs}
Alon, Friedland and Kalai proved the following corollary of the original Olson's Theorem \ref{olson} and using it in the case $p=2$, they derived the result on $p^d$-divisible subgraphs (Theorem \ref{qdiv}) mentioned in the Introduction. We present their proofs since they also show the reductions between the corresponding computational problems.

\begin{cor}[\cite{reg}] \label{corolson}
Let $n,m$ be positive integers and let $p$ be a prime. Let $d_1\geq d_2 \geq \dots \geq d_n\geq 1 $ positive integers and for $i=1,\dots ,n,j=1,\dots ,m$ let $a_{ij}$ be an integer such that $\sum_{i=1}^{n} a_{ij}$ is divisible by $p$ for every $j$ index. If $m > p^{d_n-1}-1 + \sum_{i=1}^{n-1} (p^{d_i} -1)$, then there is a subset $\emptyset \neq J \subseteq \{1,2,\dots,m \}$ such that $\sum_{j\in J} a_{ij}$ is divisible by $p^{d_i}$ for every $i=1,\dots ,n$.
\end{cor}

\begin{proof}[Proof in \cite{reg}] For $\ j=1,\dots ,m$, let $b_{ij}=a_{ij}$, if $i=1,\dots ,n-1$ and let $b_{nj}=\frac{1}{p}\left( \sum_{i=1}^{n} a_{ij} \right)$.

According to Olson's Theorem \ref{olson}, there is an $\emptyset \neq J \subseteq \{1,2,\dots,m\}$ such that $\sum_{j\in J} b_{ij}$ is divisible by $p^{d_i}$ for every $i=1,\dots ,n-1$ and $\sum_{j\in J} b_{nj}$ is divisible by $p^{d_{n-1}}$ because $m > p^{d_n-1}-1 + \sum_{i=1}^{n-1} (p^{d_i} -1)$.

However, $\sum_{j\in J} b_{nj} = \sum_{j\in J} \frac{1}{p}\left( \sum_{i=1}^{n} a_{ij} \right)$ is divisible by $p^{d_{n-1}}$ , so $\sum_{j\in J} \left( \sum_{i=1}^{n} a_{ij} \right)$ is divisible by $p^{d_{n}}$. Because of $d_1\geq d_2 \geq \dots \geq d_n\geq 1$, $\sum_{j\in J} b_{ij} = \sum_{j\in J} a_{ij}$ is divisible by $p^{d_n}$ for every $i=1,\dots ,n-1$, hence $\sum_{j\in J} a_{nj}$ should be divisible by $p^{d_n}$ and we are done. \end{proof}

\begin{theo}[Alon, Friedland, Kalai, \cite{reg}] \label{qdiv} For the maximum number of edges of a graph $G$ on $n$ vertices that contains no nontrivial $p^d$-divisible subgraph,
\[
f(n,p^d) =
\begin{cases}
 (p^d-1) \cdot n & \textnormal{ if }p\textnormal{ is an odd prime.}\\
 (2^d-1) \cdot n - 2^{d-1} & \textnormal{ if }p=2
\end{cases}
\]
\end{theo}

\begin{proof}[Proof in \cite{reg}] Here, we only prove the direction $\leq$ of the equality. In \cite{reg}, Alon et al. showed by examples that these bounds are tight.

Denote the number of edges of $G$ by $m$. Let $a_{ij}=1$ if and only if the $j^{th}$ edge is incident to the $i^{th}$ vertex, as usual. (That is, $((a_{ij}))$ is the incidence matrix of $G$.)

For an odd prime $p$, suppose $m>(p^d-1) \cdot n$. According to Olson's Theorem \ref{olson} with $d_1=d_2=\dots = d_n=d$, there is a nonempty subset $J$  of edges such that 
$\sum_{j\in J} a_{ij}$ is divisible by $p^d$ for every $i$, so there is nontrivial $p^d$-divisible subgraph.

For $p=2$, suppose $m>(2^d-1) \cdot n - 2^{d-1}= (n-1)\cdot (2^d-1) + (2^{d-1}-1) $. According to Corollary \ref{corolson} with $d_1=d_2=\dots = d_n=d$, there is a nonempty subset $J$  of edges such that $\sum_{j\in J} a_{ij}$ is divisible by $p^d$ for every $i$, so there is nontrivial $p^d$-divisible subgraph. The conditions of corollary hold, because $\sum_{i=1}^{n} a_{ij}=2$ for every index $j$. \end{proof}

%\begin{pr}[\cite{reg}]
%If $n\geq 3$, there exists graphs that has $(p^d-1) \cdot n$ edges (if $p>2$) or $(2^d-1) \cdot n -2^{d-1}$ edges (if $p=2$), but does not contain nontrivial $p^d$-divisible subgraph. Consequently,
%\[
%f(n,p^d) =
%\begin{cases}
% (p^d-1) \cdot n & \textnormal{ if }p\textnormal{ is an odd prime.}\\
% (2^d-1) \cdot n - 2^{d-1} & \textnormal{ if }p=2
%\end{cases}
%\]
%\end{pr}
%
%\begin{proof} For $p>2$ let $G_0$ be the graph which is obtained from a triangle by replacing each edges by $p^d-1$ parallel edges. For $p=2$, let $G_0$ be the similar graph, but replace one of edges by $2^{d-1}-1$ edges instead of  $2^d-1$ edges. Then add to $G_0$ $n-3$ new vertices and join each of them by $p^d-1$ edges to vertices of $G_0$.
%
%It is easy to see that this graph does not contains $p^d$-divisible subgraph, it has $n$ vertices and  $(p^d-1) \cdot n$ edges (if $p>2$) or $(2^d-1) \cdot n -2^{d-1}$ edges (if $p=2$). \end{proof}

Let us now define the computational problems corresponding to the above theorems. The existence of the solutions is guaranteed by Theorem \ref{qdiv}, Theorem \ref{olson} and Corollary \ref{corolson}, respectively.

\begin{center}
\fbox{
\parbox{0.85\linewidth}{
\smallskip

\noindent
\center
\textsc{$2^d$-divisible subgraph.}
\vspace{1mm}

\begin{tabularx}{\linewidth}{lX}
\textit{Input:}& a positive integer $d$ and a graph $G=(V,E)$, where $|V|=n$, $|E|=m$ and $m> n\cdot (2^d - 1 ) - 2^{d-1}$ holds.
\\
\textit{Find:}& a $2^d$-divisible subgraph, that is, an $\emptyset \neq F \subseteq E$ such that for every $v \in V$, the number of incident edges of $F$ is divisible by $2^d$.
\end{tabularx}
}}
\end{center}
\medskip

%For $d=1$, this is the well-known fact that in a graph, if the number of edges is at least as the number of vertices, there is an even subgraph, namely, there is a cycle in the graph. To find a cycle in such a graph, there is polynomial algorithm  but for $d>1$, we do not know whether it is in P. It could be interesting to determine the complexity of \textsc{$2^d$-divisible subgraph}.

%In \cite{reg}, one can check that the proof of Theorem \ref{qdiv} is based on Olson's theorem and Corollary \ref{corolson}, so it is also natural to study the complexity of following computational problems. That proof immediately implies Proposition \ref{graphred}.

\begin{center}
\fbox{
\parbox{0.85\linewidth}{
\smallskip

\noindent
\center
\textsc{Olson MOD $2^d$.}
\vspace{1mm}

\begin{tabularx}{\linewidth}{lX}
\textit{Input:}& 
a positive integer $d$, the integers $n$ and $m$ such that $m> n\cdot (2^d - 1 )$ and given integers $a_{ij}$ ($i=1,\dots ,n , j=1,\dots,m$).
\\
\textit{Find:}& a $\emptyset \neq J \subseteq \{1,2,\dots,m\}$ such that $\sum_{j\in J} a_{ij}$ is divisible by $2^d$ for every $i$.
\end{tabularx}
}}
\end{center}
\medskip

\begin{center}
\fbox{
\parbox{0.85\linewidth}{
\smallskip

\noindent
\center
\problem{Even-Sum Olson MOD $2^d$.}
\vspace{1mm}

\begin{tabularx}{\linewidth}{lX}
\textit{Input:}& 
a positive integer $d$, the integers $n$ and $m$ such that  $m> n\cdot (2^d - 1 ) - 2^{d-1}$ and given integers $a_{ij}$ ($i=1,\dots ,n , j=1,\dots ,m$) such that $2\mid \sum_{i=1}^n a_{ij}$.
\\
\textit{Find:}& a $\emptyset \neq J \subseteq \{1,2,\dots,m\}$ such that $\sum_{j\in J} a_{ij}$ is divisible by $2^d$ for every $i$.
\end{tabularx}
}}
\end{center}
\medskip

As we have seen in the previous sections, Olson's theorem can be proved via the Combinatorial Nullstellensatz. It implies the following propositions.

\begin{theo} \label{olsred}
\textsc{Olson MOD $2^d$} and \textsc{Even-Sum Olson MOD $2^d$} are polynomially reducible to the \textsc{Combinatorial Nullstellensatz  over $\mathbb{F}_2$}. Consequently, they are in PPA.
\end{theo}

\begin{proof}
In the proof of Theorem \ref{genolsontheorem} and Theorem \ref{maintheorem}, we construct a polynomial $f$ in $m$ variables over $\mathbb{F}_2$ such that $\deg(f)=m$ and the coefficient of $\prod_{j=1}^m x_j$ is nonzero. Such vectors $\vecc{s}$ that satisfy $f(\vecc{s})\neq 0$ precisely correspond to the subsets $\emptyset \neq J \subseteq \{1,2,\dots,m\}$ such that $\sum_{j\in J} a_{ij}$ is divisible by $2^d$ for every $i$.

We only have to check that this reduction is a polynomial reduction. In the proofs the size of the constructed polynomial is $O(2^d \cdot d \cdot n + m)$, which can be bounded $O(nm\log(m))$ due to condition $m> n\cdot (2^d - 1 )$, so the reduction is polynomial.

For reduction to the \textsc{Combinatorial Nullstellensatz  over $\mathbb{F}_2$}, we have to present a pairing function which can pair up terms $x_1x_2\dots x_m$. Here it is obvious, because there is only one term $x_1x_2\dots x_m$.

Similarly, one can check that \textsc{Even-Sum Olson MOD $2^d$} is also polynomially reducible to the \textsc{Combinatorial Nullstellensatz  over $\mathbb{F}_2$}. \end{proof}

The reduction in the proof of Theorem \ref{qdiv} immediately implies the following.

\begin{theo} \label{graphred}
\textsc{$2^d$-divisible subgraph} is polynomially reducible to \textsc{Even-Sum Olson MOD $2^d$}. Consequently, \textsc{$2^d$-divisible subgraph} is in PPA.
\end{theo}

\section{Degree-constrained subgraphs: Louigi's problem}
\label{sec:louigi}
%Similarly, one can define the computational problems \textsc{$p^d$-divisible subgraph}, \textsc{Olson MOD $p^d$} and \textsc{Combinatorial Nullstellensatz over $\mathbb{F}_p$}. However, the complexity of these problems are not known.

%Nevertheless, we present an interesting connection with Louigi's problem: we give a new proof for it, which shows that it is polynomially reducible to \textsc{$p^d$-divisible subgraph}, and hence, to \textsc{Combinatorial Nullstellensatz over $\mathbb{F}_p$}.

In Louigi's problem, given are a graph $G=(V,E)$ and forbidden sets $F(v)\subseteq \mathbb{N}$ for every $v\in V$. By an $F$-avoiding subgraph we mean a subgraph $\emptyset \neq E' \subseteq E$ such that for every $v \in V$ the number of incident edges of $E'$ is not in $F(v)$. Shirazi and Verstra\"ete \cite{louigi} proved the following theorem.
We give a new proof using our techniques.

%, which shows that it is polynomially reducible to \textsc{$p^d$-divisible subgraph}, and hence, to \textsc{Combinatorial Nullstellensatz over $\mathbb{F}_p$}.

\begin{theo}[Shirazi, Verstra\" ete \cite{louigi}] \label{louigi}
If $0 \not \in F(v)$ for all $v\in V$ and $\sum_{v\in V} |F(v)| < |E|$, then there exists a nontrivial $F$-avoiding subgraph.
\end{theo}

\begin{proof} Let $p$ be a prime greater than the maximum degree in $G$. For the node $v_i\in V $, let $Q_i=\mathbb{Z}_{p} \backslash F(v_i)$ and $a_{ij}=1$ if the node $v_i$ is incident to the edge $e_j \in E$, and 0 otherwise. Due to the conditions, $\sum_{v_i \in V} price(\mathbb{Z}_{p}\backslash Q_i) = \sum_{v_i \in V} | \mathbb{Z}_{p} \backslash Q_i| = \sum_{v\in V} |F(v)| < |E|$, so according to Theorem \ref{genolsontheorem}, there exists a subset $J$, which corresponds to a nontrivial $F$-avoiding subgraph.
\end{proof}

Note that, in \cite{louigi}, the authors also proved their theorem via the Combinatorial Nullstellensatz, but in a different way via polynomials over $\mathbb{R}$.

% \begin{cor}
% If $0 \not \in F(v)$ and $|F(v)| < \frac12 d(v)$ for all $v\in V$, then there exists a nontrivial $F$-avoiding subgraph.
% \end{cor}

% \begin{proof}
% Summing up the conditions $|F(v)| < \frac12 d(v)$ for all $v\in V$, we get that $\sum_{v\in V} |F(v)| < \frac12 \sum_{v\in V} d(v) = |E|$ and the proof is done due to Theorem \ref{louigi}.
% \end{proof}

One may ask a version of Louigi's problem modulo prime powers:  given are a prime power $p^d$, a graph $G=(V,E)$ and forbidden sets modulo $p^d$: $F(v)\subseteq \mathbb{Z}_{p^d}$ for every $v\in V$. By an $F$-avoiding subgraph modulo $p^d$ we mean a subgraph $\emptyset \neq E' \subseteq E$ such that for every $v \in V$ the number of incident edges of $E'$ is not congruent to any number in $F(v)$ modulo $p^d$. We can show the following.

\begin{theo} \label{louigimodulo}
If $0 \not \in F(v)$ for all $v\in V$ and $\sum_{v\in V} price(F(v)) < |E|$, then there exists a nontrivial $F$-avoiding subgraph modulo $p^d$.
\end{theo}

\begin{proof} Similarly to the proof above, for the node $v_i\in V $, let $Q_i=\mathbb{Z}_{p^d} \backslash F(v_i)$ and $a_{ij}=1$ if the node $v_i$ is incident to the edge $e_j \in E$, and 0 otherwise. Due to the conditions, $\sum_{v_i \in V} price(\mathbb{Z}_{p}\backslash Q_i) =  \sum_{v\in V} price(F(v)) < |E|$, so according to Theorem \ref{genolsontheorem}, there exists a subset $J$, which corresponds to a nontrivial $F$-avoiding subgraph modulo $p^d$.
\end{proof}

% \begin{cor}
% If $0 \not \in F(v)$ and $price(F(v)) < \frac12 d(v)$ for all $v\in V$, then there exists a nontrivial $F$-avoiding subgraph.
% \end{cor}

% \begin{proof}
% Summing up the conditions $price(F(v)) < \frac12 d(v)$ for all $v\in V$, we get that $\sum_{v\in V} price(F(v)) < \frac12 \sum_{v\in V} d(v) = |E|$ and the proof is done due to Theorem \ref{louigimodulo}.
% \end{proof}

Frank et al. \cite{frank} gave a polynomial time combinatorial algorithm for finding an $F$-avoiding subgraph as in Theorem \ref{louigi}. For $F$-avoiding subgraph modulo $2^d$, we show that the search problem belongs to PPA.

Let us define the corresponding computational problem. The existence of a solution is guaranteed by Theorem \ref{louigimodulo}.

\begin{center}
\fbox{
\parbox{0.85\linewidth}{
\smallskip

\noindent
\center
\textsc{Degree-constrained subgraph modulo $2^d$.}
\vspace{1mm}

\begin{tabularx}{\linewidth}{lX}
\textit{Input:}& a positive integer $d$, a graph $G=(V,E)$, subsets $F(v)\subseteq \mathbb{Z}_{p^d}$ such that $\sum_{v\in V} price(F(v)) < |E|$.
\\
\textit{Find:}& a nontrivial $F$-avoiding subgraph modulo $2^d$.
\end{tabularx}
}}
\end{center}
\medskip

% \begin{center}
% \fbox{
% \parbox{0.85\linewidth}{
% \smallskip

% \noindent
% \center
% \textsc{General Olson MOD $2^d$.}
% \vspace{1mm}

% \begin{tabularx}{\linewidth}{lX}
% \textit{Input:}& 
% a positive integer $d$, the integers $n$ and $m$, subsets  $Q_i\subseteq \mathbb{Z}_{p^d}$ such that $\sum_{i=1}^n price( \mathbb{Z}_{p^{d}} \backslash Q_i) < m$ and given integers $a_{ij}$ ($i=1,\dots ,n , j=1,\dots,m$).
% \\
% \textit{Find:}& a nonempty subset $J \subseteq \{1,2,\dots,m\}$ such that $\sum_{j\in J} a_{ij} \equiv q_i \qquad \modu{p^{d_i}} \quad \textnormal{ for some } q_i \in Q_i \textnormal{ for every }i=1,\dots ,n.$
% \end{tabularx}
% }}
% \end{center}
% \medskip

Similarly to the proofs of Theorems \ref{olsred}, \ref{graphred}, the proofs of Theorems \ref{louigimodulo}, \ref{genolsontheorem}, \ref{maintheorem} imply the following.

\begin{theo} \label{louigired}
\textsc{Degree-constrained subgraph modulo $2^d$} is polynomially reducible to the \textsc{Combinatorial Nullstellensatz over $\mathbb{F}_2$}. Consequently, \textsc{Degree-constrained subgraph modulo $2^d$} is in PPA.
\end{theo}

\section*{Acknowledgments}

I am grateful to L\'aszl\'o V\'egh for his extremely helpful comments and valuable suggestions.


\begin{thebibliography}{99}
\bibitem{alon} N. Alon:
\textit{Combinatorial Nullstellensatz}.
 Combinatorics, Probability and Computing, 8, 1999, pp. 7-29.
\bibitem{ols} J. E. Olson:
\textit{A combinatorial problem on finite Abelian groups, I}.
Journal of Number Theory, Vol. 1, Issue 1, January 1969, pp. 8-10.
\bibitem{reg} N. Alon, S. Friedland, G. Kalai:
\textit{Regular subgraphs of almost regular graphs}. Journal of Combinatorial Theory, Series B
Vol. 37, Issue 1, August 1984, pp. 79-91.
\bibitem{papa} C. H. Papadimitriou:
\textit{On the Complexity of the Parity Argument and Other Inefficient Proofs of Existence}.
Journal of Computer and System Sciences
Vol. 48, Issue 3, June 1994, pp. 498-532.
\bibitem{west} D. B. West: \textit{Research Experiences for Graduate Students in Combinatorics (Open Problem Collection)} \url{http://www.math.uiuc.edu/~west/regs/combnull.html}
\bibitem{integervalued} P-J. Cahen, J-L. Chabert: \textit{Integer-valued Polynomials}. Mathematical Surveys and Monographs 48, Providence, RI: American Mathematical Society 1997
\bibitem{louigi} H. Shirazi, J. Verstra\"ete: \textit{A note on polynomials and f-factors of graphs}. Electronic J. of Combinatorics, 15, 2008.
\bibitem{frank} A. Frank, L. C. Lau, J. Szabó: \textit{A note on degree-constrained subgraphs}.  Discrete Mathematics, 2008, 308.12, pp. 2647-2648.

\end{thebibliography}
\end{document}